\documentclass[A4,12pt]{article}
\usepackage{amsfonts,amsthm,amsmath,amssymb,euscript}
\usepackage{hyperref}
\newcommand{\sfr}[2]{\leavevmode\kern-.1em
  \raise.5ex\hbox{\the\scriptfont0 #1}\kern-.1em
  /\kern-.15em\lower.25ex\hbox{\the\scriptfont0 #2}}
\newcommand{\la}{\langle}
\newcommand{\ra}{\rangle}

\newcommand{\aut}{\mathrm{Aut}}
\newcommand{\wt}{\mathrm{wt}}

\newcommand{\supp}{\mathrm{supp}}

\newcommand{\M}{\mathbb{M}}
\newcommand{\Z}{\mathbb{Z}}

\newcommand{\al}{\alpha}
\newcommand{\be}{\beta}
\newcommand{\La}{\Lambda}
\newtheorem{thm}{Theorem}[section]
\newtheorem{prop}[thm]{Proposition}
\newtheorem{lem}[thm]{Lemma}
\newtheorem{cor}[thm]{Corollary}
\theoremstyle{remark}
\newtheorem{rem}[thm]{Remark}
\theoremstyle{definition}

\newtheorem{df}[thm]{Definition}

\newcommand{\ZZ}{\mathbb{Z}}

\newcommand{\RR}{\mathbb{R}}

\newcommand{\cC}{\mathcal{C}}
\newcommand{\EuC}{\mathcal{C}}
\newcommand{\EuD}{\EuScript{D}}

\newcommand{\nexteq}{\displaybreak[0]\\ &=}
\newcommand{\nnexteq}{\nonumber\displaybreak[0]\\ &=}
\newcommand{\allone}{\mathbf{1}}
\newcommand{\allzero}{\mathbf{0}}
\DeclareMathOperator{\Aut}{Aut}
\title{Residue codes of extremal Type~II $\Z_4$-codes and
the moonshine vertex operator algebra}
\author{
Masaaki Harada\thanks{
Department of Mathematical Sciences,
Yamagata University,
Yamagata 990--8560, Japan, and
PRESTO, Japan Science and Technology Agency, Kawaguchi,
Saitama 332--0012, Japan.
email: mharada@sci.kj.yamagata-u.ac.jp,
telephone number: +81-236-28-4533,
fax number: +81-236-28-4538},
Ching Hung Lam\thanks{
Institute of Mathematics,
Academia Sinica, Taipei 106-17, Taiwan.
email: chlam@math.sinica.edu.tw}
and Akihiro Munemasa\thanks{
Graduate School of Information Sciences,
Tohoku University,
Sendai 980--8579, Japan.
email: munemasa@math.is.tohoku.ac.jp}
}

\begin{document}

 \maketitle

\begin{abstract}
In this paper, we study the residue codes of extremal Type~II $\Z_4$-codes of
length  $24$ and their relations to the famous
moonshine vertex operator algebra. The main result
is a complete classification of all residue codes of extremal Type~II $\Z_4$-codes
of length  $24$.  Some corresponding results associated to the moonshine
vertex operator algebra are also discussed.
\end{abstract}



% \tableofcontents

\section{Introduction}

In this paper, we study the residue codes of extremal Type~II $\Z_4$-codes of
length  $24$ and their relationship to the famous moonshine
vertex operator algebra. The main
result is a complete classification of all residue codes of extremal
Type~II $\Z_4$-codes of length  $24$.
Some corresponding results about the structure codes of the moonshine
vertex operator algebra are also discussed.
Since the residue code of an extremal Type~II $\Z_4$-code
of length $24$ is contained in some binary doubly even self-dual codes
and binary doubly even
self-dual codes of length $24$ are classified in~\cite{PS75},
we can list all binary doubly
even codes $B$ satisfying the condition that its dual code
$B^\perp$ is even and $B^\perp$
has minimum weight $\geq 4$.
%% containing the all-one vector with dual codes of minimum weight $\geq 4$.
It turns out that there are 179 such codes up to equivalence
(Table~\ref{Tab:24}). Then, by using the algorithm given in~\cite{Rains}, we
determine all binary doubly even codes that can be realized as the residue codes of
some extremal Type~II $\Z_4$-codes. We also prove that if $B'\supset B$ is a
weight $4$ augmentation of $B$ (see Definition~\ref{df:wt4})
and $B$ is realized as the residue code of an extremal Type~II $\Z_4$-code, then
$B'$ is also realized (Lemma~\ref{augmentation}). Not only does this result
reduce the amount of computation, but it also helps us to  express the main
result in a nicer form (Theorem~\ref{finalresult}).

We also study  the relationship between the structure codes of the moonshine
vertex operator algebra and extremal Type~II $\Z_4$-codes of length $24$.
We call a triply even code of length $48$ a moonshine code if it is
a $\frac{1}{16}$-code of the moonshine vertex operator algebra.
The extended doubling (see Definition~\ref{Edouble})
of a binary doubly even code $B$ of length $24$ is a
triply even code of length $48$, and we show that such a code
is a moonshine code if and only if $B$
is the residue code of some extremal Type~II
$\Z_4$-code (Theorem~\ref{thm:z4-24}).
Together with our main result, this means that we know all the moonshine
codes which are extended doublings.
% if and only if its extended doubling (see Section 4  for definition) is a
%$\frac{1}{16}$ code of the Moonshine VOA.
%%
%%
%%%%%%%%%%%

The organization of the paper is as follows.
In Section~\ref{Sec:Pre}, definitions and some basic results
of codes, which are used in this paper, are given.
In Section~\ref{Sec:Enum}, we classify
residue codes of extremal Type~II $\ZZ_4$-codes of length $24$. We also show
that a binary
doubly even code is the residue code of an extremal Type~II $\ZZ_4$-code
of length $24$ if and only if it can
%all extended doublings which are moonshine codes. We show that an
%extended doubling $\EuD(B)$ is a moonshine code if and only if the code $B$ can
be obtained by successive application of weight $4$ augmentation to one of the
codes listed in Table~\ref{Tab:24-rc}.
%%%%%
In Section~\ref{sec:4},
we study the structure codes of the moonshine vertex operator algebra,
which we call moonshine codes. In particular,  we show that a binary
triply even code is a moonshine code  if and only if it can be obtained
by successive application of weight $8$ augmentation to a moonshine code
of minimum weight $16$. As a consequence, we also show that
the direct sum of the extended doublings of its components  are  moonshine
codes.% in Section~\ref{decCode}.
%%%%%
%In Appendix~\ref{Ap:MD}, we show that there exist exactly 3 extended
%%% doublings which are minimal subject to conditions (\ref{eq:c1})--(\ref{eq:c3}).
%doublings which are minimal among triply even codes $D$ of length $48$
%such that $D^\perp$ are even and have minimum weight $\geq 4$.
%% containing the all-one vector with
%% dual codes of minimum weight $\geq 4$.
%Finally, in Appendix~\ref{Ap:Z4-24}, we give explicitly generator matrices
%of several extremal Type~II $\ZZ_4$-codes,
%% ${\cC}$ with $C=\cC_1$ for each $C$ of the
%% seven codes $C=d_{12}^2$, $d_{10}e_7^2$, $d_8^3$, $d_6^4$, $d_4^6$, $e_8^3$ and
%% $d_{16}e_8$.
%in order to demonstrate that every binary doubly even self-dual code
%of length $24$ is the residue code of some extremal Type~II $\Z_4$-code.

\section{Binary codes and $\ZZ_4$-codes}
\label{Sec:Pre}

In this paper, we deal with binary codes and $\ZZ_4$-codes, and codes mean
binary codes unless otherwise specified. Let $C$ be a code of length $n$. The
{\em weight} $\wt(x)$ of a codeword $x \in C$ is the number of non-zero
coordinates. A code $C$ is called {\em even, doubly even} and {\em triply even}
if the weights of all codewords of $C$ are divisible by 2, 4 and 8,
respectively. The dual code $C^\perp$ of $C$ is defined as $\{ x \in \ZZ_{2}^n
\mid  \langle x, y\rangle = 0$ for all $y \in C\}$, where $\langle x, y\rangle$
denotes the standard inner product. A code $C$ is {\em self-orthogonal} if $C
\subset C^\perp$, and $C$ is {\em self-dual} if $C=C^\perp$. Two codes are {\em
equivalent} if one can be obtained from the other by a permutation of
coordinates. Throughout this paper, we denote the all-one vector by
$\mathbf{1}$ and the zero vector by $\mathbf{0}$.
%, and
%we denote by $x\cdot y$ the coordinatewise product $(x_1\cdot
%y_1,\dots, x_n\cdot y_n)$ of $x=(x_1,\ldots,x_n)$ and $y=(y_1,\dots, y_n) \in
%\Z_2^n$.
For a code $C$ of length $n$ and a vector $\delta\in\ZZ_2^n$, we
denote by $\langle C, \delta \rangle_{\ZZ_2}$ the code generated by the
codewords of $C$ and $\delta$.

For a $\Z_4$-code $\cC$ of length $n$, define two codes:
\[
\EuC_0=\{ \al\bmod{2}\mid \al\in\ZZ_4^n,\;
2\alpha\in \EuC\} \quad \text{ and }\quad
\EuC_1=\{ \al\bmod{2}\mid \al\in \EuC\}.
\]
These codes  $\EuC_0$ and $\EuC_1$ are called
{\em torsion} and {\em residue} codes, respectively.
It holds that $\EuC_1 \subset \EuC_0$.
For a $\ZZ_4$-code $\cC$, the dual code $\cC^\perp$ is defined
similarly to binary codes.
Then self-orthogonal codes and self-dual codes are also defined similarly.
If $\EuC$ is self-dual, then
$\EuC_1$ is doubly even and $ \EuC_0= \EuC_1^\perp$.
The {\em Euclidean weight} of a codeword $x=(x_1,\ldots,x_n)$ of $\cC$ is
$n_1(x)+4n_2(x)+n_3(x)$, where $n_{\alpha}(x)$ denotes
the number of components $i$ with $x_i=\alpha$ $(\alpha=1,2,3)$.
%% The {\em minimum Euclidean weight} $d_E$ of $\cC$ is the smallest Euclidean
%% weight among all nonzero codewords of $\cC$.
%% The Euclidean weight of a codeword $x=(x_1,\ldots,x_n)$ of $\cC$  is
%% $\sum_{i=1}^n \min\{x_i^2,(4-x_i)^2\}$.
%% A self-dual code $\cC$ is {\em Type~II}
%% if the Euclidean weights of all codewords of $\cC$ are divisible by 8.
A $\ZZ_4$-code $\cC$ is {\em Type~II}
if $\cC$ is self-dual and
the Euclidean weights of all codewords of $\cC$ are divisible by 8.
%%%%%
The {\em minimum Euclidean weight} $d_E$ of $\cC$ is the smallest Euclidean
weight among all nonzero codewords of $\cC$.
A Type~II $\ZZ_4$-code of length $n$ and
$d_E=8 \lfloor n/24 \rfloor +8$
is called {\em extremal}.
%%%%%%
Two $\ZZ_4$-codes are {\em equivalent} if one can be obtained from the
other by permuting the coordinates and (if necessary) changing
the signs of certain coordinates.

%%\begin{df}
Let $\EuC$ be a self-orthogonal $\ZZ_4$-code of length $n$.  Define
\[
A_4(\EuC)=\frac{1}2\left\{ (x_1,\dots,x_n)\in \mathbb{Z}^n \mid
(x_1 \bmod 4,\dots,x_n \bmod4) \in \EuC \right\}.
\]
It is well-known that
$A_4(\EuC)$ is even unimodular if and only if $\EuC$ is Type~II.
%%\end{df}
%%
The following result is also well-known (cf.~\cite{BSBM}).
\begin{lem}\label{lem:A4}
Let $\cC$ be a Type~II $\Z_4$-code of length $n$. 
Then, $\cC$ has
minimum Euclidean weight at least $16$ if and
only if $A_4(\cC)$ has minimum norm $4$. 
In particular, for $n=24$, 
$\cC$ is extremal if and
only if $A_4(\cC)$ is isomorphic to the Leech lattice $\Lambda$.
% Let $\cC$ be a $\Z_4$-code of length $24$. Then, $\cC$ is extremal Type~II if and
% only if $A_4(\cC)$ is isomorphic to the Leech lattice $\Lambda$.
%% Let $\cC$ be an extremal Type II $\Z_4$-code of length $24$. Then
%% $A_4(\cC)$ is isomorphic to the Leech lattice $\Lambda$.
%% \[
%% A_4(\EuC)=\frac{1}2\left\{ (x_1,\dots,x_n)\in \mathbb{Z}^n|\,
%% (x_1,\dots,x_n)\in \EuC \mod 4\right\}
%% \]
\end{lem}

%Let $C$ be a doubly even code of length $24$.
%It follows from Theorem~\ref{thm:z4-24} that if
%the doubling $\EuScript{D}(C)$ is a moonshine code,
%then $C$ is the residue code of an
%extremal Type~II $\Z_4$-code.
%% Such a code $C$ satisfies the following conditions:

\begin{lem}
Let $C$ be the residue code of an
extremal Type~II $\Z_4$-code of length $24$.
Then $C$ satisfies
the following conditions:
\begin{align}
&C\text{ is doubly even;}\label{eq:b1}\\
&C\ni\allone;\label{eq:b2}\\
&C^\perp\text{ has minimum weight at least $4$}.\label{eq:b3}
\end{align}
%Now we give a classification of codes of length $24$
%satisfying (\ref{eq:b1})--(\ref{eq:b3}).
In particular, $\dim C \ge 6$.
\end{lem}
\begin{proof}
%% Straightforward.
For the proofs of the assertions
(\ref{eq:b1}), (\ref{eq:b2}) and (\ref{eq:b3}),
see~\cite{CS}, \cite{HSG} and \cite{H}, respectively.
%% For the proofs of the assertions (\ref{eq:b1}) and (\ref{eq:b2}),
%% see \cite{CS} and \cite{HSG}, respectively.
%% If $\cC$ is an extremal Type II $\Z_4$-code
%% with $\cC_1=C$, then $\cC_0=C^\perp$, while
%% the minimum Euclidean weight of $\cC$ is at most $4$ times
%% the minimum weight of $\cC_0$.
%% This implies (\ref{eq:b3}).
It is known that a $[24,k,4]$ code exists only if 
$k \le 18$~\cite{Brouwer-Handbook}.
This gives the last assertion.
\end{proof}

\section{Classification of  residue codes of extremal Type~II $\Z_4$-codes}
\label{Sec:Enum}
%\subsection{Weight 4 augmentation and Moonshine doublings}

In this section, we classify all residue codes of extremal Type~II $\Z_4$-codes of
length $24$.

\subsection{Weight 4 augmentation}

\begin{df}\label{df:wt4}
Let $C$ be a subcode of a doubly even code $C'$ and let $k$ be
a positive integer divisible by $4$.
We call $C'$ a {\em weight $k$
augmentation of $C$} if $C'\neq C$, and $C'$ is generated by $C$ and a vector
of weight $k$.
%Let $C$ be a subcode of a doubly even code $C'$ of length $24$.
%We call $C'$ a {\em weight $4$
%augmentation of $C$} if $C'\neq C$, and $C'$ is generated by $C$ and a vector
%of weight $4$.
\end{df}

Recall that two lattices $L$ and $L'$ are neighbors if
both lattices contain a sublattice of index $2$
in common.

\begin{lem}\label{lem:1c}
Let $\La$ be an even unimodular lattice with minimum norm $4$, and suppose
that $\alpha\in\La$ satisfies $\|\alpha\|^2=4$,
where $\|\alpha\|^2=\langle \alpha,\alpha \rangle$.
Define
\begin{equation}\label{eq:11}
\La_\alpha=\{\gamma\in\La\mid\langle\alpha,\gamma\rangle\equiv0\pmod2\}.
\end{equation}
If $\beta\in\La\setminus\La_\alpha$, then
\begin{equation}\label{eq:12}
\La_{\alpha,\beta}'=\La_\alpha\cup(\frac12\alpha+\beta+\La_\alpha)
\end{equation}
is an even unimodular lattice with minimum norm $4$,
which is a neighbor of $\Lambda$ sharing $\Lambda_\alpha$.
\end{lem}
\begin{proof}
Since $\La$ is even, $\frac12\alpha\notin\La=\La^*$. This implies that
$\La_\alpha$ is
a sublattice of index $2$ in $\La$. Clearly, $\frac12\alpha\in\La^*_\alpha$, and
$\beta\in\La=\La^*\subset\La^*_\alpha$, hence $\frac12\alpha+\beta \in\La^*_\alpha$.
Since $2(\frac12\alpha+\beta)\in\La_\alpha$, we conclude that
$\La'_{\alpha,\beta}$ is a
unimodular lattice. Moreover, since
$\langle\alpha,\beta\rangle\equiv1\pmod2$, we have
$\|\frac12\alpha+\beta\|^2 \equiv0\pmod2$. Thus $\La'_{\alpha,\beta}$ is even.

It remains to show that $\La'_{\alpha,\beta}$ has minimum norm $4$. Since
$\alpha\in\La_\alpha\subset\La$, it suffices to show that
%\begin{equation}\label{eq:00}
\[
\|\frac12\alpha+\beta+\gamma\|^2\geq4
\]
%\end{equation}
for all $\gamma\in\La_\alpha$.
This follows from the inequality
\[
\|\frac12\alpha+\beta+\gamma\|^2
=\frac12\|\alpha+\beta+\gamma\|^2+\frac12\|\beta+\gamma\|^2-1
\geq3,
\]
noting that $\Lambda'_{\al,\be}$ is even.
%Suppose $\gamma\in\La_0$. Since $\beta\notin\La_0$, we have
%$\beta+\gamma\in\La\setminus\La_0$ and
%$\alpha+\beta+\gamma\in\La\setminus\La_0$. Thus
%\begin{align}
%4&\leq\|\beta+\gamma\|^2,\label{eq:01}
%\\ \intertext{and}
%4&\leq\|\alpha+\beta+\gamma\|^2\notag\\
%&=4+2(\alpha,\beta+\gamma)+\|\beta+\gamma\|^2.
%\notag
%\end{align}
%The latter implies
%\begin{equation}\label{eq:02}
%(\alpha,\beta+\gamma)\geq-\frac12\|\beta+\gamma\|^2.
%\end{equation}
%Now
%\begin{align*}
%\|\frac12\alpha+\beta+\gamma\|^2&=
%1+(\alpha,\beta+\gamma)+\|\beta+\gamma\|^2
%\\ &\geq
%1+\frac12\|\beta+\gamma\|^2
%&&\text{(by (\ref{eq:02}))}
%\\ &\geq3
%&&\text{(by (\ref{eq:01})).}
%\end{align*}
%Since $\frac12\alpha+\beta+\gamma\in\La'$ and $\La'$ is even, we obtain
%(\ref{eq:00}).
\end{proof}

The following lemma is very useful for our classification.

\begin{lem}\label{augmentation}
Let $\cC$ be a Type~II $\ZZ_4$-code of length $n$ with minimum Euclidean
weight $16$. 
%% with residue code $\cC_1$.
Let $a\in\ZZ_2^{n}$ be a
vector of weight $4$ not in $\cC_1$, such that
the code $\langle \cC_1,a\rangle_{\ZZ_2}$ is doubly even.
Then there exists a Type~II $\ZZ_4$-code 
$\cC'$ such that the minimum Euclidean weight is $16$,
$\cC'_1=\langle \cC_1,a\rangle_{\ZZ_2}$ and
$A_4(\cC')$ is a neighbor of $A_4(\cC)$.
\end{lem}
\begin{proof}
By the assumption, $\La=A_4(\cC)$ is an even unimodular lattice with minimum
norm $4$. 
% Since $\cC$ is extremal,
% by Lemma~\ref{lem:A4} the lattice $A_4(\cC)$ is
% isomorphic to the Leech lattice, so we may write
% $\Lambda=A_4(\cC)$.
%% Let $\Lambda=A_4(\cC)$. Then $\Lambda$ is isomorphic to the
%% Leech lattice.
%%, that is,  $\Lambda$ is an even unimodular
%% lattice with minimum norm $4$.
Let $e_1,\dots,e_{n}$ be the standard orthonormal basis of
$\RR^{n}$, and let
\[
\alpha=\sum_{i\in\supp(a)}e_i,
\]
where $\supp(a)$ denotes the support of $a$.
%Then $\|\alpha\|^2=4$
Since $\langle \cC_1,a\rangle_{\ZZ_2}$ is self-orthogonal,
we have
$a\in \cC_1^\perp=\cC_0$.  %%, where $\cC_0$ denote the torsion code.
Thus, $\alpha\in\La$ and $\|\alpha\|^2=4$.
Define $\La_\alpha$ by (\ref{eq:11}).
%%%%%%%
Since $a\notin \cC_1=\cC_0^\perp$, there exists $b\in \cC_0$
such that $\langle a,b \rangle=1$. Let $\be\in\ZZ^{n}$ be a vector satisfying
$\be\bmod{2}=b$. Then, $b\in\cC_0$ implies
$\be\in\Lambda$. Moreover,
$\langle a,b \rangle=1$ implies $\langle \alpha,\be \rangle
\equiv1\pmod{2}$.
Thus
Lemma~\ref{lem:1c} implies that the lattice $\La'_{\alpha,\beta}$
defined by (\ref{eq:12}) is an
even unimodular lattice with minimum norm $4$,
which is a neighbor of $\Lambda$.
%%%%%%%
Since the standard $4$-frame $\{2e_i\}_{i=1}^{n}$ is contained in
$\Lambda_\al\subset\Lambda'_{\al,\be}$,
there exists a Type~II $\ZZ_4$-code $\cC'$ such that
$A_4(\cC')=\Lambda'_{\al,\be}$.
% % % % % % 
% Since
% $\cC_1=(2\Lambda)\bmod2=(2\Lambda_\al)\bmod2$,
% we have
% %\begin{align*}
% %\cC'_1&=\cC_1\cup(a+\cC_1)
% %\nexteq
% %((2\Lambda_0)\bmod2)\cup
% %((\alpha+2\Lambda_0)\bmod2)
% %\nexteq
% %(2\Lambda')\bmod2.
% %\end{align*}
% \begin{align*}
% \cC'_1&=
% (2\Lambda'_{\al,\be})\bmod2
% \nexteq
% ((2\Lambda_\al)\bmod2)\cup
% ((\alpha+2\Lambda_\al)\bmod2)
% \nexteq
% \cC_1\cup(a+\cC_1)
% \nexteq
% \langle \cC_1,a\rangle_{\ZZ_2}.
% \end{align*}
% % % % % % 

Since
\begin{align}
\cC_1&=\{\gamma\bmod2\mid\frac12\gamma\in\La\}
\nnexteq
\{\gamma\bmod2\mid\frac12\gamma\in\La_\al\}
\cup\{\gamma \bmod2\mid\frac12\gamma\in\beta+\La_\al\}
%\nnexteq
%\{\gamma\bmod2\mid\gamma\in2\La_\al\}
%\cup\{\gamma \bmod2\mid\gamma\in2\be+2\La_\al\}
\nnexteq
\{\gamma\bmod2\mid\gamma\in2\La_\al\},
\label{eq:n1}
\end{align}
%%%%%%%
we have
\begin{align*}
\cC'_1&=\{\gamma\bmod2\mid\frac12\gamma\in\La'_{\al,\be}\}
\nexteq
\{\gamma\bmod2\mid\frac12\gamma\in\La_\al\}
\cup\{\gamma\bmod2\mid\frac12\gamma\in\frac12\alpha+\beta+\La_\al\}
&&\text{(by (\ref{eq:12}))}
\nexteq
\{\gamma\bmod2\mid\gamma\in2\La_\al\}
\cup\{\gamma\bmod2\mid\gamma\in\alpha+2\La_\al\}
\nexteq
\cC_1\cup(a+\cC_1)
&&\text{(by (\ref{eq:n1}))}
\nexteq
\langle \cC_1,a\rangle_{\ZZ_2}.
\end{align*}
Since $A_4(\cC')$ has minimum norm $4$,
$\cC'$ has minimum Euclidean weight at least $16$
by Lemma~\ref{lem:A4}.
% Since ${\cC'_0}^\perp= \cC'_1=\langle \cC_1,a\rangle_{\ZZ_2}$
% and $\cC'_1$ is self-orthogonal,
% $\cC'_0= \langle \cC_1,a\rangle_{\ZZ_2}^\perp \supset
% \langle \cC_1,a\rangle_{\ZZ_2} \ni a$.
Since $\cC'_0 \supset \cC'_1 \ni a$ and 
$\wt(a)=4$, there is a codeword of
Euclidean weight $16$ in $\cC'$.
Hence, the minimum
Euclidean weight of $\cC'$ is exactly $16$.
\end{proof}

A partial converse of the above lemma also holds.

\begin{lem}\label{-4}
Let $\cC$ be a Type~II $\ZZ_4$-code of length $n$ with minimum Euclidean
weight $16$. Suppose $a\in\cC_1$ and $\wt(a)=4$. Then there exists a Type~II
$\ZZ_4$-code $\cC'$ of length $n$ such that the 
minimum Euclidean weight is $16$,
% \begin{equation}\label{eq:0}
$\cC'_1\subsetneqq\langle\cC'_1,a\rangle=\cC_1$
% \end{equation}
and $A_4(\cC')$ is a neighbor of $A_4(\cC)$.
\end{lem}
\begin{proof}
By the assumption, $\La=A_4(\cC)$ is an even unimodular lattice with minimum
norm $4$. We may assume without loss of generality
%\begin{equation}\label{eq:1}
$a_1=1$ in $a=(a_1,\ldots,a_n)$.
%\end{equation}
Since $a\in\cC_1$, there exists
$\alpha'=(\alpha'_1,\ldots,\alpha'_n) \in\ZZ^n$ such that
$\alpha'\bmod2=a$ and
\begin{align}
\frac12\alpha'&\in\La.\label{eq:2}
%\\ \alpha'\bmod2&=a.\label{eq:3}
\end{align}
We may assume without loss of generality
%\begin{equation}\label{eq:4}
\[
\alpha'_i=\pm1\quad(i\in\supp(a)).
\]
%\end{equation}
%In view of (\ref{eq:1}), define
% Define $\alpha=(\alpha_1,\ldots,\alpha_n)\in\ZZ^n$ and $c\in\ZZ_2^n$ by
Define $\alpha=(\alpha_1,\ldots,\alpha_n)\in\ZZ^n$ by
%\begin{equation}\label{eq:5}
\[
\alpha_i=\begin{cases}
-\alpha'_1&\text{if $i=1$,}\\
\alpha'_i&\text{if $i\in\supp(a)\setminus\{1\}$,}\\
0&\text{otherwise,}
\end{cases}
\]
%\end{equation}
and set $c=\frac12(\alpha-\alpha') \bmod 2 \in\ZZ_2^n$.
Then $\alpha\in\La$, $\|\alpha\|^2=4$, and
\begin{align}
\alpha\bmod2&=\alpha'\bmod2=a,\label{eq:5x}\\
%\|\alpha\|^2&=4,\label{eq:6}\\
\langle\alpha,\alpha'\rangle&=2,\label{eq:7}\\
\langle a,c\rangle&=1.\label{eq:7ac}
\end{align}
%Moreover, since $\alpha\bmod2=a\in\cC_1\subset\cC_1^\perp=\cC_0$, we
%have
%\begin{equation}\label{eq:8}
%\alpha\in\La.
%\end{equation}
Define $\La_\alpha$ by (\ref{eq:11}), and set
%\begin{equation}\label{eq:9}
\[
\beta=-\frac12\alpha'\in\La.
\]
%\end{equation}
Then by (\ref{eq:7}), we have $\beta\in\La\setminus\La_\alpha$. Thus
Lemma~\ref{lem:1c} implies that the lattice $\La'_{\alpha,\beta}$
defined by (\ref{eq:12}) is an
even unimodular lattice with minimum norm $4$, 
which is a neighbor of $\La$.
Since $\alpha\in\ZZ^n$, the
standard $4$-frame of $\La$ is contained in $\La_\al$. This implies that there
exists a Type~II $\ZZ_4$-code $\cC'$ of length $n$ 
such that $\La'_{\al,\be}=A_4(\cC')$.
% We are now ready to show 
%It remains to prove
% (\ref{eq:0}). 

Since
%$\alpha-\alpha'\in2\ZZ^n$ by (\ref{eq:5x}), we have
%\begin{equation}\label{eq:10}
%\{\gamma\bmod2\mid\gamma\in\alpha-\alpha'+2\La_0\}
%=\{\gamma\bmod2\mid\frac12\gamma\in\La_0\}.
%\end{equation}
%Also, setting
%\begin{equation}\label{eq:11x}
%c=\frac12(\alpha-\alpha')\bmod2,
%\end{equation}
%we have, for $\frac12\gamma\in\La$,
%\begin{align}
%(\alpha,\frac12\gamma)\bmod2&=
%(\alpha-\alpha',\frac12\gamma)\bmod2
%&&\text{(by (\ref{eq:2}))}
%\nnexteq
%(\frac12(\alpha-\alpha'),\gamma)\bmod2
%\nnexteq
%(c,\gamma\bmod2).\label{eq:12x}
%\end{align}
%Thus
\begin{align}
\La_\al&=\{\frac12\gamma\in\La\mid \langle\alpha,\frac12\gamma\rangle
\equiv0\pmod2\}
&&\text{(by (\ref{eq:11}))}
\nnexteq
\{\frac12\gamma\in\La\mid \langle\alpha-\alpha',\frac12\gamma\rangle
\equiv0\pmod2\}
&&\text{(by (\ref{eq:2}))}
\nnexteq
\{\frac12\gamma\in\La\mid\langle c,\gamma\bmod2 \rangle=0\},
%&&\text{(by (\ref{eq:12x})).}
\label{eq:13}
\end{align}
we have
%Now
\begin{align*}
\cC'_1&=\{\gamma\bmod2\mid\frac12\gamma\in\La'_{\al,\be}\}
\nexteq
\{\gamma\bmod2\mid\frac12\gamma\in\La_\al\}
\cup\{\gamma\bmod2\mid\frac12\gamma\in\frac12\alpha+\beta+\La_\al\}
&&\text{(by (\ref{eq:12}))}
\nexteq
\{\gamma\bmod2\mid\frac12\gamma\in\La_\al\}
\cup\{\gamma\bmod2\mid\gamma\in\alpha-\alpha'+2\La_\al\}
%&&\text{(by (\ref{eq:9}))}
\nexteq
\{\gamma\bmod2\mid\frac12\gamma\in\La_\al\}
&&\text{(by (\ref{eq:5x}))}
%&&\text{(by (\ref{eq:10}))}
\nexteq
\{\gamma\bmod2\mid\frac12\gamma\in\La,\langle c,\gamma\bmod2\rangle=0\}
&&\text{(by (\ref{eq:13}))}
\nexteq
\{b\in\cC_1\mid \langle b,c\rangle=0\}.
\end{align*}
%Since
%\begin{align*}
%(a,c)&=(\alpha,\frac12(\alpha-\alpha'))\bmod2
%&&\text{(by (\ref{eq:5x}), (\ref{eq:11x})),}
%\nexteq
%\left(\frac12\|\alpha\|^2-\frac12(\alpha,\alpha')\right)\bmod2
%\nexteq1
%&&\text{(by (\ref{eq:6}), (\ref{eq:7})),}
%\end{align*}
%we have $a\notin\cC'_1$.
% This proves (\ref{eq:0}).
It follows from (\ref{eq:7ac}) that
$a \not \in \cC'_1$, and hence 
$\cC'_1\subsetneqq\langle\cC'_1,a\rangle=\cC_1$.

Since $A_4(\cC')$ has minimum norm $4$,
$\cC'$ has minimum Euclidean weight at least $16$
by Lemma~\ref{lem:A4}.
Since $a \in \cC_1 \subset \cC_0 \subset \cC'_0$
and $\wt(a)=4$, there is a codeword of
Euclidean weight $16$ in $\cC'$.
Hence, the minimum
Euclidean weight of $\cC'$ is exactly $16$.
\end{proof}

In Section~\ref{sec:4}, we shall give analogues of
Lemmas~\ref{augmentation} and \ref{-4} for moonshine codes.

%%%%%%%%%%%%%%%%%%%%%%%%%%%%%%%%%%%%%%
\subsection{Complete classification}\label{subsec:CC}
Here, we say that a code $C$ of length $24$ satisfying
(\ref{eq:b1})--(\ref{eq:b3}) is
{\em realizable} if $C$ can be realized as the residue code of
some extremal Type~II $\Z_4$-code. %% in other words,
%% $\EuD(C)$ is a moonshine code.

There exist nine inequivalent doubly even self-dual codes of length 
$24$~\cite{PS75}. The extended Golay code $g_{24}$ is the unique doubly even
self-dual $[24,12,8]$ code, and the other codes have minimum weight $4$ and
these codes are described by specifying the subcodes spanned by the codewords
of weight $4$, namely, $d_{12}^2,d_{10}e_7^2,d_8^3,d_6^4,d_{24},d_4^6,
e_8^3, d_{16}e_8$.
Since any code $C$ of length $24$ satisfying (\ref{eq:b1})--(\ref{eq:b3})
is contained in a doubly even self-dual code of length
$24$,
%and doubly even self-dual codes of length $24$ have been classified
%\cite{PS75},
%since every doubly even self-dual code satisfies
%(\ref{eq:b1})--(\ref{eq:b3}),
the classification of codes $C$ satisfying (\ref{eq:b1})--(\ref{eq:b3})
can be done by taking
successively subcodes of codimension $1$ starting from doubly even self-dual
codes. This method allows us to classify all codes satisfying
%% (\ref{eq:b1})--(\ref{eq:b3}). We list in Table~\ref{Tab:24} the numbers of
%% inequivalent codes of length $24$ satisfying (\ref{eq:b1})--(\ref{eq:b3}).
(\ref{eq:b1})--(\ref{eq:b3}). We list in the second column of
Table~\ref{Tab:24} the numbers
of inequivalent $[24,k]$ codes satisfying
(\ref{eq:b1})--(\ref{eq:b3}).
We remark that this classification is of independent interest, as it forms a basis
for a possible classification of extremal Type~II $\ZZ_4$-codes of length $24$.

%% {\bf ** Table 1 was moved, please check if the place is OK (H)}

%%%%%%%%%%%%%%%%%%%%%%%%%%%%%%%%%%%%%%%%%%%%
\begin{table}[bth]
\caption{Numbers of inequivalent codes of length $24$ satisfying
(\ref{eq:b1})--(\ref{eq:b3})}
\label{Tab:24}
\begin{center}
\begin{tabular}{|c|c|cc|cc|}
\hline
Dimensions $k$ & Total
&$R_{k,8}$ & $R_{k,4}$
&$N_{k,8}$ & $N_{k,4}$ \\
\hline
12& 9 & 1 &  8 & 0 &  0 \\
11&21 & 1 & 20 & 0 &  0 \\
10&49 & 3 & 44 & 0 &  2 \\
 9&60 & 6 & 40 & 4 & 10 \\
 8&32 & 4 & 16 & 8 &  4 \\
 7& 7 & 3 &  2 & 2 &  0 \\
 6& 1 & 1 &  0 & 0 &  0 \\
%% 12& 9 &   9 & 1 &  8 &  0 & 0 &  0 \\
%% 11&21 &  21 & 1 & 20 &  0 & 0 &  0 \\
%% 10&49 &  47 & 3 & 44 &  2 & 0 &  2 \\
%%  9&60 &  46 & 6 & 40 & 14 & 4 & 10 \\
%%  8&32 &  20 & 4 & 16 & 12 & 8 &  4 \\
%%  7& 7 &   5 & 3 &  2 &  2 & 2 &  0 \\
%%  6& 1 &   1 & 1 &  0 &  0 & 0 &  0 \\
\hline
\end{tabular}
\medskip

$R_{k,d}=$ the number of inequivalent realizable $[24,k,d]$ codes.

$N_{k,d}=$ the number of inequivalent non-realizable $[24,k,d]$ codes.
\end{center}
\end{table}
We use the algorithm of Rains~\cite{Rains} to determine
if a given $[24,k]$ code $C$ satisfying (\ref{eq:b1})--(\ref{eq:b3})
is realizable or not.
The algorithm is described in the form of the proof of
\cite[Theorem~3]{Rains} for classifying self-dual
$\ZZ_4$-codes,
and its modification to Type~II $\ZZ_4$-codes is straightforward.
Here, we describe the algorithm briefly.
%% Given a $[24,k]$ code $C$ satisfying (\ref{eq:b1})--(\ref{eq:b3}),
We first construct the action of the automorphism group $\Aut(C)$ of $C$
on the quotient of the
$1+\frac{k(k-1)}{2}$ dimensional space of all
Type~II $\ZZ_4$-codes $\cC$ with $\cC_1=C$, by column negations.
This defines a homomorphism from $\Aut(C)$ to $AGL(m,2)$,
where $m$ is the dimension of the quotient space, and
the orbits are in one-to-one correspondence with
equivalence classes of Type~II $\ZZ_4$-codes $\cC$ with $\cC_1=C$.
%% Hence, we find all Type~II $\ZZ_4$-codes $\cC$
%% with prescribed residue.
By enumerating orbit representatives, we obtain
all Type~II $\ZZ_4$-codes $\cC$ with $\cC_1=C$ up to equivalence.
If none of the codes $\cC$ with $\cC_1=C$ is extremal,
we conclude that $C$ is non-realizable.
%% Note that this algorithm can be used for dimensions $\dim(\cC_1) \le 10$,
%% due to the computational complexity.
%% Note that this algorithm can be implemented easily when
%% $\dim(C) \le 10$, since the maximum value of $m$ turns
%% out to be $20$ (?).
%%
This algorithm can be executed when $\dim C \le 10$, since the maximum value
of $m$ turns out to be $26$. Note that Rains~\cite[p.~220]{Rains} in 1999
commented that direct orbit finding of a $26$-dimensional matrix group is
somewhat tricky. However, with 10GB of memory, such a computation can be
done without problem nowadays. In particular, we obtain the following result.

\begin{prop}
Up to equivalence, there  is a unique extremal Type~II $\ZZ_4$-code of length
$24$ whose residue code has dimension $6$.
\end{prop}

We denote this code by
$\EuC^\natural$ and its generator matrix is given in Figure 1.

\begin{figure}[tbh] \centering {\footnotesize
\[
\begin{bmatrix}
1 1 1 1 &1 1 1 1 &1 1 1 1 &1 1 1 1& 0 0 0 0& 0 0 0 0 \\
0 2 0 0 &0 0 0 0 &1 1 1 1 &3 1 1 1& 0 0 0 2& 0 0 0 0 \\
1 1 1 1 &1 1 1 1 &0 0 0 0 &0 0 0 0& 1 1 1 1& 1 1 1 1 \\
0 1 0 0 &1 0 1 1 &1 0 1 3 &0 1 0 2& 1 0 1 1& 0 1 0 0 \\
1 1 1 0 &0 0 0 1 &1 1 3 0 &0 2 0 1& 0 0 0 1& 1 1 1 0 \\
0 1 1 1 &1 0 0 0 &3 0 2 0 &0 1 1 1& 0 1 1 1& 1 0 0 0 \\
0 0 0 0 &0 0 0 0 &2 0 2 2 &0 2 0 0& 0 0 0 0& 0 0 0 0 \\
0 0 0 0 &0 0 0 0 &2 2 2 0 &0 0 0 2& 0 0 0 0& 0 0 0 0 \\
0 0 0 0 &0 0 0 0 &2 0 0 0 &0 2 2 2& 0 0 0 0& 0 0 0 0 \\
0 2 0 0 &2 0 2 2 &0 0 0 0 &0 0 0 0& 0 0 0 0& 0 0 0 0 \\
2 2 2 0 &0 0 0 2 &0 0 0 0 &0 0 0 0& 0 0 0 0& 0 0 0 0 \\
0 2 2 2 &2 0 0 0 &0 0 0 0 &0 0 0 0& 0 0 0 0& 0 0 0 0 \\
0 2 0 0 &0 0 0 2 &2 0 2 0 &0 0 0 0& 0 0 0 0& 0 0 0 0 \\
0 2 0 0 &2 0 0 0 &2 0 0 0 &0 2 0 0& 0 0 0 0& 0 0 0 0 \\
0 2 2 0 &0 0 0 0 &2 0 0 0 &0 0 0 2& 0 0 0 0& 0 0 0 0 \\
0 2 0 0 &0 0 0 2 &0 0 0 0 &0 0 0 0& 0 0 0 2& 0 2 0 0 \\
0 2 0 0 &2 0 0 0 &0 0 0 0 &0 0 0 0& 0 0 2 2& 0 0 0 0 \\
0 2 2 0 &0 0 0 0 &0 0 0 0 &0 0 0 0& 0 0 0 2& 2 0 0 0
%% changed (2011/9)
\end{bmatrix}
\]
\caption{A generator matrix of the extremal Type~II $\Z_4$-code $\EuC^\natural
$} \label{fig:z4} }
\end{figure}

When $\dim C=11$ or $12$,
the value of $m$ ranges from $33$ to $46$, and a direct method will fail.
However, we randomly found an extremal Type~II $\ZZ_4$-code $\cC$
with $\cC_1=C$
without finding all inequivalent Type~II $\ZZ_4$-codes.
The two doubly even self-dual codes with labels $g_{24}$ and $d_{24}$ can be
realized as the residue codes of some extremal Type~II $\ZZ_4$-codes of length
$24$~\cite{CS97}.
%Moreover, Young and Sloane
%found an extremal Type II $\ZZ_4$-code $\EuC$ with $C={\EuC}_1$ for each $C$ of
%the remaining seven codes (cf.~\cite[Postscript]{CS97}).
%%% \end{rem}
%%% \begin{rem}\label{rem:24-2}
%%% The above result by Young and Sloane described in \cite[Postscript]{CS97} does
%The above result does
%not appear in the literature. In Appendix~\ref{Ap:Z4-24},
%we give such an extremal Type~II
%$\ZZ_4$-code.
%% $\EuC$ with $C={\EuC}_1$ for each $C$ of the remaining seven codes.
Moreover, Young and Sloane claim to have found an extremal Type II
$\ZZ_4$-code $\EuC$ with $C={\EuC}_1$ for each $C$ of the remaining seven
codes (cf.~\cite[Postscript]{CS97}), although no explicit information about such
codes has been published since then. In Appendix~\ref{Ap:Z4-24}, we give such
an extremal Type~II $\ZZ_4$-code for each of the seven doubly even self-dual
codes.
In particular,
our result for the case $\dim C=12$
confirms the claim in~\cite[Postscript]{CS97}.

%% due to the computational complexity.
%In this way, we were able to determine the realizability of
%all codes in Table~\ref{Tab:24}.
%%
%% For a $[24,k]$ code $C$ satisfying (\ref{eq:b1})--(\ref{eq:b3})
%% with $k=11$ and $12$,
%% using the search method described in \cite{Z4-PLF}
%% (see also Appendix \ref{Ap:Z4-24}),
%% we have found an extremal Type~II $\ZZ_4$-code $\cC$ with $\cC_1=C$.
%% In this way, using the two methods,
%% %% we were able to determine whether all codes in Table~\ref{Tab:24}
%% %% are realizable or not.
%% we were able to determine the realizability of all codes in Table~
%% \ref{Tab:24}.

%% \begin{rem}
%% In Table \ref{Tab:24-nrc}, we give the dimension $m$ and the number $N$
%% of orbits for some non-realizable codes.
%% %% It turns out that the ten codes given in Table \ref{Tab:24-nrc}
%% %% are maximal non-realizable codes.
%% \textbf{We describe why the ten codes are considered}
%% \end{rem}

%%%%%%%%%%%%%

In Table~\ref{Tab:24}, we list the number $R_{k,d}$ of
inequivalent realizable $[24,k,d]$ codes.
All such codes with $d=8$
are listed in Table~\ref{Tab:24-rc},
where $C_6={\EuC^\natural}_1$
and
$C_{7,1},C_{7,2}$
are defined in Appendix~\ref{subsec:OMS},
and the codes other than $C_6,C_{7,1},C_{7,2}$ are
generated by the code $C$ and the vectors $v$ listed in
Table~\ref{Tab:24-rv}.
Note that $C_6,C_{7,1},C_{7,2}$ are
minimal subject to (\ref{eq:b1})--(\ref{eq:b3}).
Also, in Table~\ref{Tab:24}, we list the number $N_{k,d}$
of inequivalent non-realizable $[24,k,d]$ codes.
Maximal codes among these codes are listed in Table~\ref{Tab:24-nrc},
where the codes are generated by $C_6$ and the vectors $v$ listed in
Table~\ref{Tab:24-nrv}.
All other non-realizable codes can be obtained from a maximal
one; see Theorem~\ref{finalresult}~(iii) below.
Also, in Table~\ref{Tab:24-nrc}, we give the dimension $m$
of the quotient space and the number $N$
of the equivalence classes of Type~II $\ZZ_4$-codes $\cC$
with $\cC_1=C$
for the maximal non-realizable codes $C$.
%% In Table \ref{Tab:24-nrc}, we also give the dimension $m$
%% of the quotient space and the number $N$
%% of the orbits for the maximal non-realizable codes.
Generator matrices of all codes in Table~\ref{Tab:24}
can be obtained electronically from

%\begin{verbatim}
\url{http://www.math.is.tohoku.ac.jp/~munemasa/de24extremalresidue.htm}
%\end{verbatim}

The following is the main theorem of the paper.

\begin{thm}\label{finalresult}
Let $C$ be a doubly even code of length $24$ containing $\allone$ such
that $C^\perp$ has minimum weight at least $4$. Then the following are
equivalent.
\begin{enumerate}
%\item[(i)] there exists an extremal Type II $\ZZ_4$-code
%$\cC$ with $\cC_1=C$;
\item[(i)] the code $C$ is the residue code of an extremal Type~II $\Z_4$-code;
\item[(ii)] successive applications of weight $4$ augmentation to
one of the codes in Table~\ref{Tab:24-rc} gives a code equivalent to $C$;
\item[(iii)] none of the codes in Table~\ref{Tab:24-nrc} can be
obtained by successive applications of weight $4$ augmentation to
a code equivalent to $C$.
\end{enumerate}
\end{thm}
\begin{proof}
%% We can enumerate all doubly even codes of length $24$ containing the all-one
%% vector such that the minimum weight of its dual code is at least $4$.
%% As described in Subsection \ref{Subsec:24}, we classified
%% codes of length $24$ satisfying
%% the conditions (\ref{eq:b1})--(\ref{eq:b3}).
%% Then we can classify
%% these doubly even codes into those which satisfy (i) and those which
%% do not. In particular, we can check that all the codes in Table~\ref{Tab:24-rc}
%% satisfy (i), and that none of the codes in Table~\ref{Tab:24-nrc} satisfies
%% (i). Also, we can classify all the pairs $(C,C')$ such that $C$ is a weight
%As described above, we determined all realizable codes
%and non-realizable codes.
%4 We can classify all the pairs $(C,C')$ such that $C$  a weight
%4 $4$ augmentation of $C'$. This allows us to verify the implication
%By Theorem~\ref{thm:z4-24}, the condition (i) is equivalent to $C$ being
%realizable. Thus,
The implications (ii)$\implies$(i)$\implies$(iii) follow from
Lemma~\ref{augmentation}. The implication (iii)$\implies$(ii) can be verified by
classifying all the pairs $(S,C)$ such that $C$ is a weight $4$ augmentation of a
subcode $S$ of $C$ of codimension $1$.
%We can classify all the pairs $(S,C)$
%such that $C$ is a weight $4$ augmentation of
%a subcode $S$ of $C$ of codimension $1$.
%This allows us to verify the implication
%(iii)$\implies$(ii). The implications (ii)$\implies$(i)$\implies$(iii) follow
%from Lemma~\ref{augmentation}.
\end{proof}

\begin{rem}
The implication (i)$\implies$(ii) also follows from Lemma~\ref{-4}.
\end{rem}

%% In the process of proving the above theorem,
%% we have the following.
%%
%% \begin{thm}\label{thm:z4-24-9}
%% Up to equivalence,
%% there is a unique extremal Type~II $\ZZ_4$-code of length $24$
%% whose residue code has dimension $6$.
%% \end{thm}
%%
%% Hence, every extremal Type~II $\ZZ_4$-code of length $24$ whose
%% residue code has dimension $6$ is equivalent to
%% $\EuC^\natural$ given in Section \ref{Sec:MVOA}.
%In the process of proving Theorem~\ref{finalresult},
%we obtain the uniqueness of the code
%$\EuC^\natural$ given in Section \ref{Sec:MVOA}.

%\begin{prop}
%Up to equivalence,
%$\EuC^\natural$ is a unique extremal Type~II $\ZZ_4$-code of
%length $24$
%whose residue code has dimension $6$.
%\end{prop}

%%%%%%%%%%%%%%%%%%%%%%%%%%%%%%%%%%%%%%%%%%%%%%%%
\begin{table}[thbp]
\caption{Realizable codes with minimum weight $8$} \label{Tab:24-rc}
\begin{center}
{\small
%{\footnotesize
%{\scriptsize
\begin{tabular}{|l|l|l|}
%% \noalign{\hrule height0.8pt}
\hline
\multicolumn{1}{|c|}{Codes} &
\multicolumn{1}{c|}{$C$}   & \multicolumn{1}{c|}{$v$} \\
\hline
$C_6$      &  & \\
$C_{7,1 }$ &  & \\
$C_{7,2 }$ &  & \\
$C_{7,3 }$ & $C_{6   }$ & $v_{7  }$           \\
$C_{8,1 }$ & $C_{7,3 }$ & $v_{81 }$           \\
$C_{8,2 }$ & $C_{7,3 }$ & $v_{82 }$           \\
$C_{8,3 }$ & $C_{7,3 }$ & $v_{83 }$           \\
$C_{8,4 }$ & $C_{6   }$ & $v_{841}$, $v_{842}$\\
$C_{9,1 }$ & $C_{8,3 }$ & $v_{91 }$           \\
$C_{9,2 }$ & $C_{8,4 }$ & $v_{92 }$           \\
\hline
%% \noalign{\hrule height0.8pt}
   \end{tabular}
   \begin{tabular}{|l|l|l|}
\hline
%% \noalign{\hrule height0.8pt}
\multicolumn{1}{|c|}{Codes} &
\multicolumn{1}{c|}{$C$}   & \multicolumn{1}{c|}{$v$} \\
\hline
$C_{9,3 }$ & $C_{7,3 }$ & $v_{931}$, $v_{932}$\\
$C_{9,4 }$ & $C_{8,3 }$ & $v_{94 }$           \\
$C_{9,5 }$ & $C_{8,1 }$ & $v_{95 }$           \\
$C_{9,6 }$ & $C_{8,3 }$ & $v_{96 }$           \\
$C_{10,1}$ & $C_{9,4 }$ & $v_{101}$           \\
$C_{10,2}$ & $C_{9,4 }$ & $v_{102}$           \\
$C_{10,3}$ & $C_{9,4 }$ & $v_{103}$           \\
$C_{11  }$ & $C_{10,1}$ & $v_{11 }$           \\
$C_{12  }$ & $C_{11  }$ & $v_{12 }$           \\
\ && \\
%% \noalign{\hrule height0.8pt}
\hline
   \end{tabular}
}
\end{center}
\end{table}
%%%%%%%%%%%%%%%%%%%%%%%%%%%%%%%%%%%%%%%%%%%%%%%%

%%%%%%%%%%%%%%%%%%%%%%%%%%%%%%%%%%%%%%%%%%%%%%%%
\begin{table}[thbp]
\caption{Vectors for realizable codes} \label{Tab:24-rv}
\begin{center}
{\small
%{\footnotesize
%{\scriptsize
\begin{tabular}{|c|c|}
%% \noalign{\hrule height0.8pt}
\hline
\multicolumn{2}{|c|}{Vectors} \\
\hline
$v_{7  }$ & $(1,0,0,0,0,1,1,1,0,0,1,1,0,0,1,1,0,1,1,0,0,1,1,0)$\\
$v_{81 }$ & $(0,1,0,0,0,0,0,1,0,0,0,0,1,0,1,0,0,0,1,1,0,0,1,1)$\\
$v_{82 }$ & $(1,0,0,0,0,1,1,1,0,0,1,0,1,1,0,1,0,0,1,1,0,0,1,1)$\\
$v_{83 }$ & $(1,0,0,0,0,1,1,1,0,0,1,0,1,1,0,1,0,0,0,0,0,0,0,0)$\\
$v_{841}$ & $(0,0,0,0,1,1,0,0,0,0,0,1,1,1,0,1,0,1,0,0,0,1,0,0)$\\
$v_{842}$ & $(0,0,0,0,1,1,1,1,0,0,0,0,1,1,0,0,0,0,1,1,0,0,0,0)$\\
$v_{91 }$ & $(0,1,0,0,0,0,0,1,0,0,0,0,1,1,1,1,0,0,1,0,1,0,0,0)$\\
$v_{92 }$ & $(1,0,0,0,0,1,0,0,0,0,1,0,0,0,1,0,0,1,0,1,1,0,0,1)$\\
$v_{931}$ & $(0,0,0,0,1,0,0,1,0,0,1,0,0,0,0,1,0,0,0,1,1,0,1,1)$\\
$v_{932}$ & $(0,0,0,0,1,1,0,0,0,0,0,0,1,1,1,1,0,1,1,1,1,0,1,1)$\\
$v_{94 }$ & $(0,1,0,0,0,0,1,0,0,0,0,0,0,0,1,1,0,1,1,1,0,0,1,0)$\\
$v_{95 }$ & $(1,0,0,0,0,1,1,1,0,0,1,0,1,1,0,1,0,0,1,1,0,0,1,1)$\\
$v_{96 }$ & $(1,0,0,0,0,1,1,1,0,0,0,0,0,0,0,0,0,0,1,1,0,0,1,1)$\\
$v_{101}$ & $(1,0,0,0,0,1,1,1,0,0,0,0,0,0,0,0,0,0,1,1,0,0,1,1)$\\
$v_{102}$ & $(1,0,0,0,0,0,1,0,0,0,1,0,0,0,0,1,0,0,0,1,0,1,1,1)$\\
$v_{103}$ & $(1,0,0,0,0,0,0,1,0,0,1,0,1,0,0,0,0,1,1,0,0,1,0,1)$\\
$v_{11 }$ & $(1,0,0,0,0,0,1,0,0,0,0,0,1,0,0,1,0,1,0,0,0,1,1,1)$\\
$v_{12 }$ & $(1,0,0,0,0,0,0,1,0,0,0,0,0,1,0,1,0,0,0,1,1,1,0,1)$\\
\hline
%% \noalign{\hrule height0.8pt}
   \end{tabular}
}
\end{center}
\end{table}
%%%%%%%%%%%%%%%%%%%%%%%%%%%%%%%%%%%%%%%%%%%%%%%%

%%%%%%%%%%%%%%%%%%%%%%%%%%%%%%%%%%%%%%%%%%%%%%%%
\begin{table}[thbp]
\caption{Maximal non-realizable codes} \label{Tab:24-nrc}
\begin{center}
{\small
%{\footnotesize
%{\scriptsize
\begin{tabular}{|l|l|l|c|c|}
%% \noalign{\hrule height0.8pt}
\hline
\multicolumn{1}{|c|}{Codes} &
\multicolumn{1}{c|}{$C$}   & \multicolumn{1}{c|}{$v$}
& $m$ & $N$ \\
\hline
%% $D_{7  }$ & $C_{6 }$ & $w_{7  }$ \\
%% $D_{81 }$ & $D_{7 }$ & $w_{81 }$ \\
%% $D_{82 }$ & $D_{7 }$ & $w_{82 }$ \\
$N_{9,1 }$ & $C_{6}$ & $w_7,w_{81},w_{91 }$ & 14 & 159 \\
$N_{9,2 }$ & $C_{6}$ & $w_7,w_{81},w_{92 }$ & 14 & 372 \\
$N_{9,3 }$ & $C_{6}$ & $w_7,w_{81},w_{93 }$ & 14 & 170 \\
$N_{9,4 }$ & $C_{6}$ & $w_7,w_{82},w_{94 }$ & 14 & 388 \\
$N_{9,5 }$ & $C_{6}$ & $w_7,w_{82},w_{95 }$ & 14 & 228 \\
\hline
%% \noalign{\hrule height0.8pt}
\end{tabular}
\begin{tabular}{|l|l|l|c|c|}
\hline
%% \noalign{\hrule height0.8pt}
\multicolumn{1}{|c|}{Codes} &
\multicolumn{1}{c|}{$C$}   & \multicolumn{1}{c|}{$v$}
& $m$ & $N$ \\
\hline
$N_{9,6 }$ & $C_{6}$ & $w_7,w_{82},w_{96 }$     & 14 & 254 \\
$N_{9,7 }$ & $C_{6}$ & $w_7,w_{82},w_{97 }$     & 14 & 287 \\
$N_{9,8 }$ & $C_{6}$ & $w_7,w_{82},w_{98 }$     & 14 & 488 \\
%%% $D_{9  }$ & $D_{81}$ & $w_{9  }$ \\
$N_{10,1}$ & $C_{6}$ & $w_7,w_{81},w_9,w_{101}$ & 23 & 299 \\
$N_{10,2}$ & $C_{6}$ & $w_7,w_{81},w_9,w_{102}$ & 23 & 378 \\
\hline
%% \noalign{\hrule height0.8pt}
   \end{tabular}
}
\end{center}
\end{table}
%%%%%%%%%%%%%%%%%%%%%%%%%%%%%%%%%%%%%%%%%%%%%%%%

%%%%%%%%%%%%%%%%%%%%%%%%%%%%%%%%%%%%%%%%%%%%%%%%
\begin{table}[thbp]
\caption{Vectors for non-realizable codes} \label{Tab:24-nrv}
\begin{center}
{\small
%{\footnotesize
%{\scriptsize
\begin{tabular}{|c|c|}
%% \noalign{\hrule height0.8pt}
\hline
\multicolumn{2}{|c|}{Vectors} \\
\hline
$w_{7  }$& $(0,0,0,0,0,1,1,0,0,0,1,1,0,0,0,0,0,0,0,1,1,0,1,1)$ \\
$w_{81 }$& $(0,0,0,0,0,0,1,1,0,0,0,1,1,1,1,0,0,1,1,1,1,0,1,1)$ \\
$w_{82 }$& $(0,0,0,0,0,0,1,1,0,1,1,1,0,1,0,0,0,1,1,1,0,1,1,1)$ \\
$w_{91 }$& $(0,0,0,0,1,0,0,1,0,0,1,1,0,1,0,1,0,1,0,0,0,1,0,0)$ \\
$w_{92 }$& $(0,0,0,0,1,0,0,1,0,0,0,1,0,0,0,1,0,1,0,1,0,0,1,1)$ \\
$w_{93 }$& $(0,0,0,0,0,0,0,0,0,1,0,1,0,0,0,0,0,0,0,0,0,1,0,1)$ \\
$w_{94 }$& $(0,0,0,0,0,0,0,0,0,1,0,0,0,0,1,0,0,0,0,0,1,0,0,1)$ \\
$w_{95 }$& $(0,0,0,0,0,0,0,0,0,0,0,1,0,1,0,0,0,0,0,1,0,1,0,0)$ \\
$w_{96 }$& $(0,0,0,0,0,0,0,0,0,0,1,1,1,0,0,1,0,1,1,1,0,0,1,0)$ \\
$w_{97 }$& $(0,0,0,0,0,0,0,0,0,1,1,1,1,0,0,0,0,0,1,0,1,1,0,1)$ \\
$w_{98 }$& $(0,0,0,0,0,0,0,0,0,1,0,1,0,0,0,0,0,1,0,1,0,0,0,0)$ \\
$w_{9  }$& $(0,0,0,0,1,0,0,1,0,1,0,0,1,0,1,1,0,0,1,0,0,1,0,0)$ \\
$w_{101}$& $(0,0,0,0,0,0,0,0,0,1,0,1,0,1,1,0,0,1,1,0,0,1,0,1)$ \\
$w_{102}$& $(0,0,0,0,0,0,0,0,0,1,0,0,1,0,1,1,0,0,0,0,0,0,0,0)$ \\
%% \noalign{\hrule height0.8pt}
\hline
   \end{tabular}
}
\end{center}
\end{table}

%%%%%%%%%%%%%%%%%%%%%%%%%%%%%%%%%%%%%%%%%%%%%%%%

%\newpage
%%%%%%%%%%%%%%%%%%%%%%%%%%%%%%%%%%%%%%%%%%

\section{Moonshine vertex operator algebra and its structure codes}
\label{sec:4}

In this section, we study the relationship between the moonshine
vertex operator algebra and
extremal Type~II $\Z_4$-codes.  Our notations for vertex operator algebras (VOA)
and  framed VOAs are standard. We shall refer to~\cite{FLM,DGH98,LY} for details.

%%%%%%%%%%%%%%%
\subsection{Moonshine codes and extremal Type~II $\ZZ_4$-codes}
Recall that  the moonshine VOA $V^\natural$ is constructed
by~\cite{FLM} 
as a $\Z_2$-orbifold of the Leech lattice VOA $V_\Lambda$. Namely,
\begin{equation}\label{Vn}
V^\natural=\tilde{V}_\Lambda=
(V_\Lambda)^\theta \oplus (V_\Lambda^T)^\theta,
%=\tilde{V}_\Lambda(\theta),
\end{equation}
where $\theta$ is an automorphism of $V_\Lambda$ lifted by the
$(-1)$-isometry of the Leech lattice $\Lambda$,
%(see Remark \ref{theta}),
$V_\Lambda^T=V_\Lambda^T(\theta)$ is the unique irreducible
$\theta$-twisted module for $V_\Lambda$  and $(V_\Lambda^T)^\theta$ is the
submodule fixed by $\theta$ (see~\cite{FLM}).
%
%\begin{rem}
It was shown in~\cite{DMZ} that $V^\natural$ is a framed VOA, i.e., $V^\natural$
contains a subVOA $T$, called  a frame, which is isomorphic to the tensor
product of $48$ copies of the simple Virasoro VOA $L(\frac{1}2,0)$.
%\end{rem}

Given a holomorphic framed VOA $V$ and a frame $T$, one
can associate a triply even code  $D$, called the structure code or the
$\frac{1}{16}$-code, of $V$  (cf.~\cite{DGH98,LY}).  Let $C=D^\perp$.  Then  $V$
can be decomposed as $V =\oplus_{\be \in D} V^\be$  such that $V^\be, \be \in
D,$ are irreducible $V^0$-modules and $V^0$ is isomorphic to the code
VOA $M_C$ as constructed in~\cite{M2}. However,  the structure code $D$
depends on the choice of $T$ and there are many possible choices for the frame
$T$ in general. The main purpose of this section is to study the
structure codes of the moonshine VOA $V^\natural$.

\begin{df}
We call a triply even code of length $48$ a
%% \textit{moonshine code}  if it can be realized as the
%% $\frac{1}{16}$-code of $V^\natural$.
\textit{moonshine code}  if it can be realized as a $\frac{1}{16}$-code of
$V^\natural$ with respect to some Virasoro frame.
\end{df}

\begin{lem}[\cite{DGH98,M3}]\label{mcode}
Let $D$ be a moonshine code. Then $D$ satisfies the following conditions:
\begin{align}
&D\text{ is triply even,}\label{eq:c1}\\
&D\ni\allone,\label{eq:c2}\\
&D^\perp\text{ has minimum weight at least $4$}.\label{eq:c3}
\end{align}
Moreover, $\dim D\geq 7$.
%%where $\allone$ denotes the all-one vector.
%% \begin{enumerate}
%% \item $D$  is triply even;
%% \item $D \ni\allone$;
%% \item $C=D^\perp$  has minimum weight at least $4$.
%% \end{enumerate}
\end{lem}
%\begin{proof}
%For the proofs of the assertions (\ref{eq:c1}) and (\ref{eq:c2}), see \cite{LY} (see
%also Theorem~\ref{thm:LY}). The assertion (\ref{eq:c3}) is obtained, since
%$(V^\natural)_1=0$, where $V^\natural=\bigoplus_{n=0}^\infty (V^\natural)_n$.
%This is because
%\[
%\sum_{i=1}^n h_i=\min\{m\mid U\cap (V^\natural)_m\neq0,\;
%U\cong L(h_1,\dots,h_n)\}.
%\]
%% The assertion (\ref{eq:c3}) is obtained from that
%% $D$ is a structure code of $V^\natural$
%% (see the above argument in this subsection).
%It is known that a $[48,k,4]$ code exists only if $k \le 41$
%\cite{Brouwer-Handbook,DGH98}. This gives the last assertion.
%\end{proof}

Now, let us define two linear maps $d, \ell :\Z_2^n \to \Z_2^{2n}$ such that
\begin{equation*}\label{mapdlr}
\begin{split}
d(a_1,a_2, \dots, a_n)& = (a_1,a_1, a_2,a_2, \dots, a_n,a_n),\\
\ell(a_1,a_2, \dots, a_n)& =(a_1,0, a_2,0, \dots, a_n,0),\\
%r(a_1,a_2, \dots, a_n)& =(0,a_1,0, a_2,\dots,0, a_n),
\end{split}
\end{equation*}
for any $(a_1,a_2, \dots, a_n)\in \Z_2^n$.

\begin{df}\label{Edouble}
 Let $C$ be a code of length $n$. We define
%% $$\EuScript{D} (C) = \la  d(C), (1,0)^n\ra_{\ZZ_2}  $$
$$\EuScript{D} (C) = \la  d(C), \ell(\allone)\ra_{\ZZ_2}  $$
to be the code generated by $d(C)=\{d(x) \mid x \in C\}$
and $\ell(\allone)$. We call the code
$\EuScript{D}(C)$  the \textit{extended doubling  (or simply the doubling)} of
$C$.
\end{df}

%\begin{thm}[\cite{LY}]\label{thm:2.14}
%Suppose that $\EuC$ is a Type II $\ZZ_4$-code of length $n$.
%% Let $\cC$ be a Type II $\ZZ_4$-code of length $n$.
%Then the VOA $\tilde{V}_{A_4(\EuC)}$ is a holomorphic framed VOA. Moreover,
%the
%% The VOA $\tilde{V}_{A_4(\EuC)}$ is a holomorphic framed VOA. Moreover, the
%structure codes with respect to the frame $T= (V_{\sqrt{2}A_1}^+)^{\otimes n}$
%for $\tilde{V}_{A_4(\EuC)}$ are given by
%$(\EuScript{D}(\cC_1)^\perp,\EuScript{D}(\cC_1))$.
%\end{thm}

%Finally, we give some basic facts on extended doublings.

\begin{lem}
\label{lem:DT} If $C$ is a doubly even $[8n,k]$ code, then the extended doubling
$\EuScript{D}(C)$ is a triply even $[16n,k+1]$ code.
%% Moreover, if $C$ contains the all-one vector of length $8n$
%% and $C^\perp$ has minimum weight
%% $\ge 4$, then $\EuScript{D}(C)$
%% contains the all-one vector of length $16n$ and
%% $\EuScript{D}(C)^\perp$ has minimum weight $\ge 4$.
Moreover, if $C^\perp$ is even and has minimum weight $\ge 4$, then
$\EuScript{D}(C)^\perp$ is even and has minimum weight $\ge 4$.
\end{lem}
\begin{proof}
Straightforward.
\end{proof}

The following result is  essentially proved in~\cite{DGH98} (see also~\cite{LY}).
%\begin{thm}\label{thm:A4}
%Let $\cC$ be a $\Z_4$-code of length $24$. Then, $\cC$ is extremal Type II if and
%only if $A_4(\cC)$ is isomorphic to the Leech lattice $\Lambda$.
%% Let $\cC$ be an extremal Type II $\Z_4$-code of length $24$. Then
%% $A_4(\cC)$ is isomorphic to the Leech lattice $\Lambda$.
%% \[
%% A_4(\EuC)=\frac{1}2\left\{ (x_1,\dots,x_n)\in \mathbb{Z}^n|\,
%% (x_1,\dots,x_n)\in \EuC \mod 4\right\}
%% \]
%\end{thm}

%Thus, by Theorem~\ref{thm:2.14} and \eqref{Vn}, we have the following
%proposition.
\begin{prop}\label{prop:z4-24-1}
Let $\cC$ be an extremal Type~II $\Z_4$-code of length $24$. Then the doubling
$\EuScript{D}(\cC_1)$ can be realized as a $\frac{1}{16}$-code of the
moonshine VOA $V^\natural$.
%%, where $B=\EuC_1$ is the residue code of $\EuC$.
\end{prop}
%%%%%% proof for Prop 3.5 %%%%%%%
%\begin{proof}
%Since $\cC$ is an extremal Type~II $\Z_4$-code, $A_4(\EuC)$ is isomorphic to
%$\Lambda$ by Theorem~\ref{thm:A4}. Hence,
%$\tilde{V}_{A_4(\EuC)}\cong\tilde{V}_\Lambda=V^\natural$ by \eqref{Vn}.
%Recall from  \eqref{Vn} that
%\[
%%% V^\natural= (V_\Lambda)^\theta \oplus (V_\Lambda^T)^\theta
%V^\natural =\tilde{V}_\Lambda(\theta).
%\]
%Since $\cC$ is an extremal Type~II $\Z_4$-code, $A_4(\EuC)$ is isomorphic to
%$\Lambda$ by Theorem \ref{thm:A4}. Hence,
%%by Remark \ref{gandtheta},
%we have
%$V^\natural\cong \tilde{V}_{A_4(\EuC)}(\theta )= \tilde{V}_{A_4(\EuC)}$.
%Now, by Theorem~\ref{thm:2.14}, the structure codes for $V^\natural\cong
%\tilde{V}_{A_4(\EuC)}$ are given by $(\EuScript{D}(\cC_1)^\perp,
%\EuScript{D}(\cC_1))$.
%\end{proof}
%%%%%% end %%%%%%%%%%%%%%%%%%%%

%\begin{proof}
%Straightforward.
%\end{proof}

The converse also holds. First let us recall a  $\Z_2$-orbifold construction for
holomorphic framed VOAs (see Theorem~8 of \cite{LY}).

\begin{lem}[\cite{LY}] \label{taud}
Let $V=\bigoplus_{\beta \in D} V^\beta$ be a framed VOA and $C=D^\perp$. Let
$\delta\in \Z_2^n\setminus C$ be a vector of even weight
and  denote $ D^0=\{\be\in D
\mid \la \be, \delta\ra=0\}$. Then
\[
\tilde{V}(\delta)=
\bigoplus_{\be\in D^0} \Big( V^\be  \oplus  (M_{\delta+C}\times_{M_C} V^\be) \Big)
\]
is also a holomorphic framed VOA and $D^0$ is the $\frac{1}{16}$-code, where $\times_{M_C}$ denotes
the fusion product with respect to the VOA $M_C$.
%The construction of $\tilde{V}(\delta)$ is often referred as a $\Z_2$-orbifold construction.
\end{lem}

The following is the main theorem of this section.

\begin{thm}
\label{thm:z4-24} Let $B$ be a doubly even code of length $24$.
%satisfying (\ref{eq:b1})--(\ref{eq:b3}).
Then, the doubling $\EuD(B)$ is a moonshine code if and only if there exists an
extremal Type~II $\ZZ_4$-code $\cC$ of length $24$ with $B={\cC }_1$.
\end{thm}

\begin{proof}
The ``if'' part follows from Proposition~\ref{prop:z4-24-1}. To prove the
converse, suppose that $D=\EuD(B)$ is a moonshine code.

Let $\delta=(1,1,0,0, \dots,0) \in \ZZ_2^{48}$ and $C=D^\perp$. Since the
minimum weight of $C$ is at least $4$ by Lemma~\ref{mcode}, we have
$\delta\notin C$.  By Lemma~\ref{taud}, the  $\frac{1}{16}$-code of the
$\Z_2$-orbifold VOA $\tilde{V}^\natural(\delta)$  is  $D^0 =\{\be\in \EuD(B)
\mid \la \be, \delta\ra=0\}= d(B)$. Since $D^\perp =\la \ell(B^\perp), d(\la
\allone\ra_{\ZZ_2}^\perp)\rangle_{\ZZ_2}$, we have $ (D^0)^\perp \supset \la
d(\la \allone\ra_{\ZZ_2}^\perp), \delta \ra_{\ZZ_2}=d(\Z_2^{24}). $

Thus, $\tilde{V}^\natural(\delta)$ contains a subalgebra $M_{d(\Z_2^{24})}$,
which is isomorphic to the lattice VOA $V_{\sqrt{2}A_1}^{\otimes 24}$
(cf.~\cite{LY}). Hence, $\tilde{V}^\natural(\delta)$ must be isomorphic to a lattice
VOA $V_L$ for some even unimodular lattice $L$ (cf.~\cite{Dong}), and
$L/(\sqrt{2}A_1^{\oplus 24})$ defines a Type~II $\Z_4$-code $\EuC$ with
$\EuC_1= B$. By~\cite[Proposition~4.2]{LY2} (see also~\cite{DM04}),
%By the same
%argument as in \cite{LY2} (see also \cite{DM04}),
$\tilde{V}^\natural( \delta)$ is isomorphic to the Leech lattice VOA
$V_\Lambda$. Therefore, $L$ is isomorphic to $\Lambda$ and $\EuC$ is an
extremal Type~II $\Z_4$-code by Lemma~\ref{lem:A4}.
%Since   $V_{\sqrt{2}A_1}^{\otimes 24}$ decomposes as $M_{d(\Z_2^{24})}$ with
%respect to any Virasoro frames in  $V_{\sqrt{2}A_1}^{\otimes 24}$
%(cf.~\cite{M2}), we may assume without loss that $T$ is given as  in
%\cite{DMZ}.
\end{proof}

Together with Theorem~\ref{finalresult}, we can determine all the moonshine
codes which are extended doublings.

%%%%%%%%%%%%% Section 5 %%%%%%%%%%%%%%%%
\subsection{Weight 8 augmentation and other moonshine codes}
\label{decCode}

In this subsection, we shall give analogues of 
Lemmas~\ref{augmentation} and \ref{-4} for moonshine codes.

Recall  that the full automorphism group $\aut(V^\natural)$ of $V^\natural$ is
the Monster simple group $\M$ and it has two conjugacy classes of involutions
denoted by $2A$ and $2B$ in~\cite{atlas}. Their $\Z_2$-twisted modules
$V^T(2A)$ and $V^T(2B)$ were constructed in~\cite{La00} 
and \cite{Hua}, respectively. Their minimal weights
are also determined.
\begin{lem}\label{2A2B}
\begin{itemize}
\item[(i)]
The minimal weight of the $2A$-twisted module $V^T(2A)$ is $\frac{1}2$, and\\
$\dim (V^T(2A))_{1/2}=1$.
\item[(ii)]
The minimal weight of the $2B$-twisted module $V^T(2B)$ is $1$, and
$\dim(V^T(2B))_{1}=24$.
\end{itemize}
\end{lem}

The structure of the corresponding $\Z_2$-orbifold VOA is also determined.

\begin{lem}[\cite{Hua,La00,LY,Ya}]\label{2A2BVOA}
Let $g$ be an involution of $\aut(V^\natural)= \M$. 
% Then the $\Z_2$-orbifold VOA
% \[
% \tilde{V^\natural}(g)\cong
% \begin{cases}
% V^\natural & \text{ if $g$ belongs to $2A$},\\
% V_\Lambda  & \text{ if $g$ belongs to $2B$},\\
% \end{cases}
% \]
% where two VOAs $V,V'$, which are isomorphic,
% are denoted by $V \cong V'$.
Then the $\Z_2$-orbifold VOA
$\tilde{V^\natural}(g)$ is isomorphic to
$V^\natural$ if $g$ belongs to $2A$,
$V_\Lambda$ if $g$ belongs to $2B$.
\end{lem}

The next theorem is an analogue of Lemma~\ref{augmentation}
for moonshine codes.

\begin{thm}\label{+8}
Suppose that
$D$ is a moonshine code. Let $\xi\in \Z_2^{48}\setminus D$ be such that
$D'=\la D,\xi \ra_{\ZZ_2}$ is triply even. If $\xi+D$ has minimum weight $8$, then
$D'$ is also a moonshine code.
\end{thm}

\begin{proof}
%% Let $C=D^\perp$, so that $(C,D)$ is the structure codes of
%% $V^\natural$ with respect to a Virasoro frame $T$. Let
Let $D$ be the $\frac{1}{16}$-code of $V^\natural$ with respect to a Virasoro
frame $T$ and $C=D^\perp$. Let $V^\natural= \bigoplus_{\be\in D}V^\be$ and
%Let $V^\natural= \bigoplus_{\be\in D}V^\be$. Denote  $C=D^\perp$ and
$C^0=\{\al\in C\mid \la \al, \xi\ra =0\}$.
Suppose that $\xi+D$ has minimum weight $8$. Without loss
of generality, we may assume that
$\xi=(\xi_1, \dots, \xi_{48})$ has weight $8$.
%We shall construct a $\Z_2$-twisted module using $\xi$ and the
%construction in  \eqref{2-twisted} such that its minimal weight is as small as
%possible.

Set
\[
h^\xi_i=
\begin{cases}
0 &\text{ if } \xi_i=0,\\
\frac{1}{16}  &\text{ if } \xi_i=1,
\end{cases}
\]
and let $U$ be an irreducible $M_{C^0}$-module which contains
$\bigotimes_{i=1}^{48} L(\frac{1}2, h^\xi_i)$ as a $T$-submodule. Then the
minimal weight of $U$ is $\frac{1}{16} \wt(\xi)=\frac{1}2$.
By~\cite[Theorem~1]{LY}, there exists an automorphism $g\in \aut(V)$ of order
$2$ such that $g(V^\be)=V^\be$ for all $\be \in D$. Let $V^{\be, \pm}$ be the $\pm 1$-eigenspaces of $g$ on $V^\be$. Then $V^\beta \times_{M_{C^0}} U = (V^{\beta,+} \times_{M_{C^0}} U) \oplus (V^{\beta,-} \times_{M_{C^0}} U)$ is a sum of two irreducible $M_{C^0}$-modules. One has 
weights in $\Z$ and the other has weights in $\frac{1}2+\Z$ (cf.~\cite{La00,LY}). Moreover, the $M_{C^0}$-module
\begin{equation}\label{2-twisted}
V^T(g)=\bigoplus_{\beta\in D} V^\beta \times_{M_{C^0}} U
\end{equation}
forms an irreducible  $g$-twisted module of $V$.  For each $\be \in D$, let
$U^\be= V^{\be, +}$ and let $U^{\xi+\be}$
be the integral part of $V^\beta \times_{M_{C^0}} U$. Then by  Theorem~8 
of~\cite{LY},
\[
\tilde{V}^\natural(g)  =\bigoplus_{\beta\in D}  \big( U^\beta \oplus U^{\xi+\be}\big)
\]
is a holomorphic framed VOA whose $\frac{1}{16}$-code is $D'=\la D,
\xi\ra_{\Z_2}$.

Since $U$ is isomorphic to 
$M_{C^0}\times_{M_{C^0}}U \subset V^0\times_{M_{C^0}} U$,
the minimal weight of $(V^\natural)^T(g)$ is $\leq \frac {1}2$. Thus, the minimal
weight of $(V^\natural)^T(g)$ is $\frac {1}2$ and $(V^\natural)^T(g)$ is a
$2A$-twisted module by Lemma~\ref{2A2B}. Therefore, the $\Z_2$-orbifold
VOA $\tilde{V}^\natural (g) $ is isomorphic to $V^\natural$ by
Lemma~\ref{2A2BVOA} and we have the desired conclusion.
\end{proof}

By Lemma~\ref{lem:DT}, doublings $\EuD(e_8)$, $\EuD(d_{16}^+)$ and
$\EuD(e_8\oplus e_8)$ are triply even.
%, where $e_8\oplus e_8$ is
%the unique decomposable doubly even self-dual code of length $16$.
Thus, $\EuD(e_8)\oplus \EuD(e_8) \oplus \EuD(e_8)$, $\EuD(e_8\oplus e_8)
\oplus \EuD(e_8)$ and $\EuD(d_{16}^+)\oplus \EuD(e_8)$ are also triply even
codes of length $48$.  By Theorem~\ref{+8}, we also have the following result.

%% Cor. 5.4
\begin{prop}
The triply even codes $\EuD(e_8)\oplus \EuD(e_8) \oplus \EuD(e_8)$,
$\EuD(e_8\oplus e_8) \oplus \EuD(e_8)$ and $\EuD(d_{16}^+)\oplus \EuD(e_8)$
are moonshine codes.
\end{prop}

\begin{proof}
We note that
\begin{align*}
\EuD(e_8\oplus e_8)\oplus \EuD(e_8)
&=\la \EuD(e_8 \oplus e_8 \oplus e_8), ((0,0)^{16}, (1,0)^8)\ra_{\ZZ_2},
\\
\EuD(e_8)\oplus \EuD(e_8) \oplus \EuD(e_8)
&= \la \EuD(e_8 \oplus e_8)\oplus \EuD(e_8),
((1,0)^8, (0,0)^{16}) \ra_{\ZZ_2},
%\EuD(e_8)\oplus \EuD(e_8) \oplus \EuD(e_8)
%&= \la \EuD(e_8 \oplus e_8\oplus e_8),
%((1,0)^8, (0,0)^{16}), ((0,0)^{16}, (1,0)^8)\ra_{\ZZ_2},
\\
 \EuD(d_{16}^+)\oplus \EuD(e_8)
&=\la \EuD(d_{16}^+ \oplus e_8), ((0,0)^{16}, (1,0)^8)\ra_{\ZZ_2}.
\end{align*}
%% By Corollary~\ref{SDcode},
The doublings $\EuD(e_8 \oplus e_8\oplus e_8)$ and
$\EuD(d_{16}^+ \oplus e_8)$ of the two decomposable doubly even self-dual
codes of length $24$ are moonshine codes (see Table~\ref{Tab:24}). 
Hence, we have the desired result by
Theorem~\ref{+8}.
\end{proof}

\begin{rem}
The three codes above have dimensions greater than $13$, while the doubling of
any doubly even self-dual code has dimension $13$ by Lemma~\ref{lem:DT}.
Hence, none of the three codes is equivalent to any extended doubling of a
doubly even self-dual code.
\end{rem}

The next theorem is a partial converse of Theorem~\ref{+8}, which can also be
viewed as an analogue of Lemma~\ref{-4} for moonshine codes.

\begin{thm}\label{-8}
Let $D$ be a moonshine code. Suppose $\eta \in D$
and $\wt(\eta)=8$.
Then there exists a moonshine code $D'$ such that $D' \subsetneqq\langle D',
\eta  \rangle=D.$
\end{thm}

\begin{proof}
Let $D$ be the $\frac{1}{16}$-code of $V^\natural$ with respect to a frame $T$
and $C=D^\perp$. Then $V^\natural= \bigoplus_{\be\in D}V^\be$ and $V^\be,
\be \in D,$ are irreducible modules for the corresponding code VOA $V^0=M_C$.

Set
\[
h^\eta_i=
\begin{cases}
0 &\text{ if } \eta_i=0,\\
\frac{1}{16}  &\text{ if } \eta_i=1.
\end{cases}
\]
Let $U$ be an irreducible $M_{C}$-module which contains
$\bigotimes_{i=1}^{48} L(\frac{1}2, h^\xi_i)$ as a $T$-submodule such that the
fusion product $V^\eta \times_{M_{C}} U$ has integral weights.  Note that
$V^\eta \times_{M_{C}} U$ is isomorphic to $M_{\al +C}$ for some $\al\in
\Z_2^{48}$. Since the minimal weight of $U$ is $\frac{1}{16}
\wt(\eta)=\frac{1}2$,  we have  $\langle\al, \eta\rangle\neq 0$ and hence $\al
\notin C$. In this case, we have a $\Z_2$-twisted module
\[
(V^\natural)^T =\bigoplus_{\beta\in D} V^\beta \times_{M_{C}} U = \bigoplus_{\beta\in D} V^\beta \times_{M_{C}} M_{\al+C}.
\]
By  Lemma~\ref{2A2B}, $(V^\natural)^T$ is a $2A$-twisted module of
$V^\natural$.

Now set $D'=\{\be \in C\mid \la \al, \be\ra =0\}$. Then $ D' \subsetneqq\langle
D', \eta  \rangle=D.$   By Lemmas~\ref{taud} and \ref{2A2BVOA},
the $\Z_2$-orbifold VOA
\[
\tilde{V}(\al)  = \bigoplus_{\be\in D'}   \left (  V^\be +  V^\be \times_{M_C}  M_{\al+C}\right)
\]
is isomorphic  to the moonshine VOA $V^\natural$.  Therefore, $D'$ is a
moonshine code.
\end{proof}

%% \begin{rem}
%% In Theorem~\ref{+8}, the special case where $D=\EuD(C)$ for a
%% code $C$ satisfying (\ref{eq:b1})--(\ref{eq:b3}) and $\xi=\ell(\eta)$
%% with $\wt(\eta)=4$, follows from Theorems~\ref{thm:A4}, \ref{thm:z4-24}
%% and Lemma~\ref{augmentation}.
%% \end{rem}
%\begin{rem}
%In Theorem~\ref{+8}, the special case where $D=\EuD(C)$ for a
%code $C$ satisfying (\ref{eq:b1})--(\ref{eq:b3}) and $\xi=d(a)$
%for a vector $a\in\ZZ_2^{24}$ of weight $4$,
%follows from Theorem~\ref{thm:z4-24}
%and Lemma~\ref{augmentation}.
%\end{rem}

\begin{cor}\label{pm8}
A triply even code $D$ of length $48$ is a moonshine  code if and only if $D$ can
be obtained by successive weight $8$ augmentations from a moonshine code with
minimum weight $16$.
\end{cor}

\begin{rem}
After this work has been completed, triply even codes of length $48$ were classified by
Betsumiya and Munemasa~\cite{BM}. Their work is, in some sense,
complementary to ours. In particular,  the classification of moonshine
codes can be reduced to checking the realizability of few triply even
codes with minimum weight $16$
which are not doublings based on their work and 
Theorems~\ref{+8} and \ref{-8} in this work.
\end{rem}

%% %%%%%%%%%%%%%%%%%
%% \begin{table}[bht]
%% \caption{Numbers of inequivalent codes of length $24$ satisfying
%% (\ref{eq:b1})--(\ref{eq:b3})} \label{Tab:24S}
%% \begin{center}
%% \begin{tabular}{|c|c|c|}
%% \hline
%% Dimensions & \# codes& \# realizable codes \\
%% %& & residue codes of extremal $\Z_4$codes\\
%% \hline
%% 12 &  9 &  9 \\
%% 11 & 21 &  \\
%% 10 & 49 &  \\
%% 9  & 60 & 46 \\
%% 8  & 32 & 20 \\
%% 7  &  7 &  5 \\
%% 6  &  1 &  1 \\
%% \hline
%% \end{tabular}
%% \end{center}
%% \end{table}
%% %%%%%%%%%%%%%%%%%%%%%%%%%%%%%%%%%%%%%%%%%%%%
%%%%%%%%%%%%% Section 5 %%%%%%%%%%%%%%%%

%%%%%%%%%%%%%%%%%%%%%%%%%%%%%%%%%%%%%%%%
\appendix
\section{Appendix}
\subsection{Extremal Type~II $\ZZ_4$-codes of length 24
whose residue codes are doubly even self-dual codes}
\label{Ap:Z4-24}

Here,
%% In this appendix,
%as described in Remark \ref{rem:24},
we give explicitly extremal Type~II
$\ZZ_4$-codes ${\cC}$ of length $24$
with $C=\cC_1$ for each $C$ of the seven
doubly even self-dual
codes with labels $d_{12}^2$, $d_{10}e_7^2$, $d_8^3$, $d_6^4$, $d_4^6$,
$e_8^3$ and $d_{16}e_8$.
This is done by listing their generator matrices
$\begin{bmatrix}
I_{12} & M \\
\end{bmatrix}$
where $M$ are as follows:

{\small
\[
\begin{bmatrix}
311222000022\\
112302000002\\
310010020022\\
130221220020\\
330202300020\\
121311130200\\
323111323220\\
103111322100\\
101133102210\\
101331322003\\
132202213331\\
213131331333
\end{bmatrix},
\begin{bmatrix}
333220002022\\
132120002220\\
132012002022\\
332023220200\\
002220333202\\
213113033002\\
213113101200\\
233113312000\\
220002220131\\
121311022231\\
321111002301\\
123133202312
\end{bmatrix},
\begin{bmatrix}
131000222202\\
132122020002\\
110232020022\\
220001310020\\
202001321022\\
013310313220\\
231133213000\\
121331120100\\
101131102032\\
301111300221\\
130221300133\\
231330002113
\end{bmatrix},
\begin{bmatrix}
311222200002\\
130102222200\\
020213122002\\
220013010200\\
033301131202\\
213321132320\\
312230333300\\
121311003120\\
121130132010\\
303110112021\\
112203132233\\
031131002233
\end{bmatrix},
\]
\[
\begin{bmatrix}
311002220220\\
200333022220\\
220022311222\\
011031123102\\
121330332122\\
332121211120\\
301123121212\\
031110033232\\
312231130030\\
101231031201\\
310310303221\\
033123130021
\end{bmatrix},
\begin{bmatrix}
213120000200\\
103102202202\\
132100002020\\
131022002202\\
222001310220\\
002012130200\\
220011032000\\
022211320222\\
202002022131\\
000020203233\\
222200023303\\
022202201310
\end{bmatrix}
\text{ and }
\begin{bmatrix}
211302200000\\
103102022022\\
332320022222\\
333020220220\\
022001331331\\
200230113111\\
220013033113\\
220231321111\\
002231110131\\
022013313031\\
220231133321\\
002233113310
\end{bmatrix},
\]
}
respectively.

%%%%
\subsection{Two $[24,7]$ codes $C_{7,1}$ and $C_{7,2}$}
\label{subsec:OMS}
Up to equivalence, there exist two $[24,7]$ codes which are minimal
subject to (\ref{eq:b1})--(\ref{eq:b3}) (see Subsection~\ref{subsec:CC}).
Here, we give generator matrices of the two
$[24,7]$ codes.
Since these two codes along with $C_6$
are used to define
other codes in Tables~\ref{Tab:24-rc} and \ref{Tab:24-nrc},
we define the two codes by fixing the coordinates.

% As described in Subsection~\ref{Subsec:24}, we have classified
% codes of length $24$ satisfying
% the conditions (\ref{eq:b1})--(\ref{eq:b3}).
% In particular, we have verified that
% there exist exactly three codes which are minimal subject to
% the conditions (\ref{eq:b1})--(\ref{eq:b3}).
% %% We give descriptions of these three codes.
% Since these three codes are used to define
% other codes in Tables
% \ref{Tab:24-rc} and \ref{Tab:24-nrc},
% we define the three codes by fixing the coordinates.
%
% The unique $[24,6]$ code $C_6$ satisfying the conditions
% (\ref{eq:b1})--(\ref{eq:b3}) has generator matrix
% \[
% \begin{bmatrix}
% M_8&M_8&M_8\\
% \allone&\allzero&\allzero\\
% \allzero&\allone&\allzero
% \end{bmatrix},
% \]
% where
% $M_8=
% \begin{bmatrix}
% I_4 & J-I_4\\
% \end{bmatrix}$,
% $J$ denotes the all-one $4 \times 4$ matrix,
% and $I_n$ denotes the identity matrix of order $n$.
% Note that $M_8$ is a generator matrix of $e_8$.
% %% the extended Hamming $[8,4,4]$ code.
% %% We have verified that
% %% the doubling $\EuScript{D}(C_{6})$ is equivalent to
% %% the unique $[48,7]$ code $D^\natural$.
% %%
% Up to equivalence, there exist two $[24,7]$ codes which are minimal
% subject to (\ref{eq:b1})--(\ref{eq:b3}).
The first one $C_{7,1}$ has generator matrix
\[
\begin{bmatrix}
M_6&M_6&M_6&M_6\\
\allone&\allone&\allzero&\allzero\\
\allone&\allzero&\allone&\allzero
\end{bmatrix},
\]
where $M_6$ denotes a generator matrix of the parity check $[6,5,2]$
code $E_6$.
%% \begin{verbatim}
%% E:=GeneratorMatrix(EvenWeightCode(6));
%% j:=KMatrixSpace(K,1,6)![1,1,1,1,1,1];
%% z:=KMatrixSpace(K,1,6)!0;
%% M1:=VerticalJoin(<
%%    HorizontalJoin([E,E,E,E]),
%%    HorizontalJoin([j,j,z,z]),
%%    HorizontalJoin([j,z,j,z])>);
%% C71:=LinearCode(M1);
%% \end{verbatim}
In order to construct the second one, we first construct a $6\times 12$
bordered double circulant matrix
\[
M_{12}=\begin{bmatrix}
1&\allzero &0&\allone\\
\allzero^T&I_5&\allone^T&C_5\end{bmatrix}
=
\begin{bmatrix}
100000011111\\
010000101001\\
001000110100\\
000100101010\\
000010100101\\
000001110010
\end{bmatrix},
\]
where $C_5$ denotes the adjacency matrix of a $5$-cycle.
%% and $I_n$ denotes the identity matrix of order $n$.
%% Note that the double
%% circulant code with generator matrix $M$ is not self-dual.
The second one $C_{7,2}$ has generator matrix
\[
\begin{bmatrix}
\allone&\allzero
\\ M_{12}&M_{12}
\end{bmatrix}.
\]

\bigskip
\noindent
{\bf Acknowledgments.}
The research for this paper was partially carried out
while the first and the third authors were visiting
Institute of Mathematics, Academia Sinica, Taiwan.
These authors would like to thank Academia Sinica
for the hospitality during this visit.

%%%%%%%%%%%%%%%%%  References  %%%%%%%%%%%%%%%%%%%%%%%%

\end{document}